\documentclass[a4paper,12pt]{amsart}
\usepackage{amsfonts,amsmath,amssymb,amsthm,mathtools}
\usepackage{bbm}
\usepackage{geometry} 
\usepackage{hyperref}
\usepackage{fixmath} 
\usepackage{paralist} 

\usepackage{graphicx}

\hypersetup{
pdftitle={Perturbing transient RWRE with cookies of maximal strength},
pdfsubject={Mathematics},
pdfauthor={Elisabeth Bauernschubert},
pdfkeywords={Excited random walk in a random environment, cookies of strength 1, recurrence, transience, subcritical branching process in a random environment with immigration.}
}

\theoremstyle{plain}
\newtheorem{theorem}{Theorem}[section]
\newtheorem*{theorem_RWRE}{Theorem RWRE}
\newtheorem*{theorem_ERW}{Theorem ERW}
\newtheorem{lemma}[theorem]{Lemma}
\newtheorem{proposition}[theorem]{Proposition}
\theoremstyle{remark}
\newtheorem{remark}[theorem]{Remark}
\newtheorem{example}[theorem]{Example}
\theoremstyle{definition}
\newtheorem{Definition}[theorem]{Definition}
\newtheorem{assumptionletter}{Assumption}

\newcommand{\N}{\mathbb{N}}
\newcommand{\Z}{\mathbb{Z}}
\newcommand{\R}{\mathbb{R}}

\DeclareMathOperator{\Var}{Var}

\begin{document}

\title[Perturbing transient RWRE with cookies of maximal strength]{Perturbing transient Random Walk in a Random Environment with cookies of maximal strength}

\author{Elisabeth Bauernschubert}
\address{Mathematisches Institut\\
    Eberhard Karls Universit\"at T\"ubingen\\
    Auf der Morgenstelle 10\\
    72076 T\"ubingen, Germany}
\email{elisabeth.bauernschubert@uni-tuebingen.de}

\date{\today}


\begin{abstract}
We consider a left-transient random walk in a random environment on $\Z$ that will be disturbed by cookies inducing a drift to the right of strength 1. The number of cookies per site is i.i.d.\ and independent of the environment.
Criteria for recurrence and transience of the random walk are obtained. For this purpose we use subcritical branching processes in random environments with immigration and formulate criteria for recurrence and transience for these processes.
\end{abstract}

\maketitle

\section{Introduction}
We investigate random walks with random transition probabilities. Therefore,
set $\Omega:=([0,1]^{\N})^{\Z}$ with elements $\omega=((\omega(x,i))_{i\geq 1})_{x\in\Z}$. Suppose that elements from $\Omega$ are chosen at random according to a probability measure $\mathbb{P}$ on $\Omega$ with corresponding expectation operator $\mathbb{E}$. For fixed environment $\omega\in\Omega$ and $z\in\Z$ define a nearest-neighbour random walk $(S_n)_{n\geq0}$ on a suitable probability space $\Omega'$ with probability measure $P_{z,\omega}$, which satisfies:
\begin{align*}
P_{z,\omega}[S_0=z]&=1,\\
P_{z,\omega}[S_{n+1}=S_n+1|(S_m)_{1\leq m\leq n}]&= \omega(S_n,\#\{m\leq n:\: S_m=S_n\}),\\
P_{z,\omega}[S_{n+1}=S_n-1|(S_m)_{1\leq m\leq n}]&= 1-\omega(S_n,\#\{m\leq n:\: S_m=S_n\}).
\end{align*}
In this way $\omega(x,i)$ is the transition probability from $x$ to $x+1$ upon the $i^{th}$ visit at site $x$. Furthermore let us denote the so-called \emph{annealed} or \emph{averaged} probability measure by $P_x[\cdot]:=\mathbb{E}[P_{x,\omega}[\cdot]]$, with corresponding expectation operator $E_x$. The process $(S_n)_{n\geq0}$ is called \emph{recurrent} (\emph{transient}) if $S_n=0$ infinitely often ($\lim_{n\to\infty}S_n\in\{\pm \infty\}$) $P_0$-a.s.

In the case where for $\mathbb{P}$-a.e.\ $\omega$ there exists a sequence $(p_z)_{z\in\Z}\in [0,1]^{\Z}$ such that $\omega(z,i)=p_z$ for all $i\geq 1$, $z\in\Z$, and $(p_z)_{z\in\Z}$ i.i.d.\ under $\mathbb{P}$, $(S_n)_{n\geq0}$ is called a one-dimensional \emph{random walk in a random environment} (RWRE). In general, we refer the reader to \cite{Solomon75, Zeitouni} for results and information about RWRE, but give here the recurrence/transience criteria for RWRE found in \cite[Theorem (1.7)]{Solomon75}.
\begin{theorem_RWRE}[Solomon 1975]
 \label{th:RWRE} 
Consider an RWRE $(S_n)_{n\geq0}$ and assume that $0\leq p_z < 1$ for all $z\in\Z$ or $0< p_z\leq 1$ for all $z\in\Z$. Let $\mathbb{E}[\log \rho_0]$ be well defined, where $\rho_0:=(1-p_0)p_0^{-1}$.
\begin{enumerate}
\item[(i)] If $\mathbb{E}[\log \rho_0]<0$, then $\lim_{n\to\infty}S_n = +\infty$ $P_0$-a.s.
 \item[(ii)] \label{RWRE_2} If $\mathbb{E}[\log \rho_0]>0$, then $\lim_{n\to\infty}S_n = -\infty$ ${P_0}$-a.s.
 \item[(iii)] If $\mathbb{E}[\log \rho_0]=0$, then $-\infty=\liminf_{n\to\infty}S_n < \limsup_{n\to\infty}S_n= \infty$ ${P_0}$-a.s.
\end{enumerate}
\end{theorem_RWRE}

We can see that the RWRE is a special case of a multi-excited (cookie) random walk (ERW) with infinitely many cookies per site. In the ERW-model $\omega(z,i)$ is not restricted to be constant in $i$ for all $z\in\Z$ $\mathbb{P}$-a.s.
Excited random walks go back to the work of Benjamini and Wilson in \cite{Benjamini}. Further studies of these processes and an extension to multi-excited random walks have been made by Zerner in \cite{Zerner05, Zerner06}, by Basdevant/Singh in \cite{Basdevant08b, Basdevant08a} and by Kosygina/Zerner in \cite{Zerner08}. In \cite[Theorem 12]{Zerner05} Zerner proves the following recurrence and transience criteria.

\begin{theorem_ERW}[Zerner 2005]
 Assume that $\omega\in ([{1}/{2},1]^{\N})^{\Z}$ $\mathbb{P}$-a.s.\ and $(\omega(x,\cdot))_{x\geq 0}$ is stationary and ergodic under $\mathbb{P}$. Then, $(S_n)_{n\geq0}$ is recurrent if and only if $\mathbb{E}[\sum_{i\geq 1}(2 \omega(0,i)-1)]\leq 1$.
\end{theorem_ERW}

In \cite{Zerner08} Kosygina and Zerner discussed an ERW with a bounded number of cookies, i.e.\ $\omega(z,i)={1}/{2}$ for all $i>K$ for all $z$ $\mathbb{P}$-a.s.\ for some constant $K$, and showed the following in \cite[Theorem 1]{Zerner08}.
\begin{theorem_ERW}[Kosygina, Zerner 2008]
 Let $K\in\N$ and $\omega(z,i)={1}/{2}$ for all  $i>K$ for all $z$ $\mathbb{P}$-a.s. Assume that $(\omega(z,\cdot))_{z\in\Z}$ is i.i.d., $\mathbb{E}[\prod_{n=1}^K{\omega(0,n)}]>0$ and $\mathbb{E}[\prod_{n=1}^K{(1-\omega(0,n))}]>0$.
If $\delta:=\mathbb{E}[\sum_{n=1}^K{(2\omega(0,n)-1)}]\in[-1,1]$ then $(S_n)_{n\geq0}$ is recurrent. If $\delta>1$ ($\delta<-1$ respectively) then $(S_n)_{n\geq0}$ is transient to the right (left respectively).
\end{theorem_ERW}

The recurrence and transience criteria for the RWRE and for the ERW seem to be quite different. So the challenging question arises whether one can find a unifying criterion. In the present paper we begin with a small step in our undertaking and consider the following combination of these processes which we will call \emph{excited random walk in a random environment} (ERWRE for short). For an illustration of the model discussed in this paper see Figure~\ref{fig:model}.

\begin{figure}[!ht]%
\centering

\scalebox{0.65}{\includegraphics[width=\columnwidth]{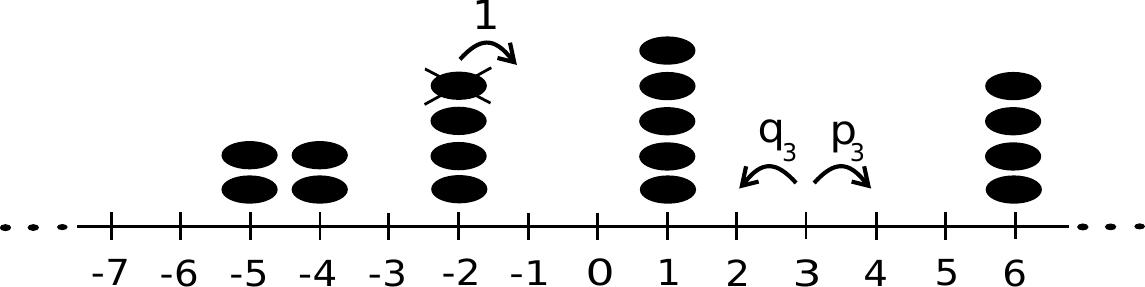}}
\caption{Model of the random walk.}
{\parbox{0.85\textwidth}{\small If there are cookies at his current position $z\in\Z$, the random walker removes one and makes a step to $z+1$. If there is no cookie he jumps to the right with probability $p_z$ and to the left with probability $q_z:=1-p_z$.}}
\label{fig:model}
\end{figure}

Consider an environment $(p_z)_{z\in\Z}\in (0,1)^{\Z}$ and put $M_z$ cookies on every integer $z\in\Z$. Now a nearest neighbour random walk $(S_n)_{n\geq0}$ is started at 0 with the following transition probabilities.
If the random walker comes to site $z$ and if there is still at least one cookie on this site, he removes one cookie and jumps to $z+1$. Otherwise he makes a step to the right with probability $p_z$ and to the left with probability $q_z:=1-p_z$.\\
The cookies in our model have maximal strength and induce a drift to the right. On the other hand,  an environment $(p_z)_{z\in\Z}$ is assumed that makes an RWRE (i.e.\ without cookies) tend to $-\infty$. So jump- and cookie-environment cause a drift in opposite directions and the question naturally arises which drift is stronger in certain cases. We found criteria for transience
and recurrence of the process.

Let us define the number of cookies of strength 1 at site $x\in\Z$ by
\begin{align*}
&M_x:=0 \qquad\text{if }\quad\{i\geq 1:  \; \omega(x,j)=1\;\forall 1\leq j\leq i\}=\emptyset ,\\
&M_x:=\sup\{i\geq 1:  \; \omega(x,j)=1\;\forall 1\leq j\leq i\} \qquad\text{otherwise.}
\end{align*}

For further discussions we postulate the following.
\begin{assumptionletter}
         \label{as_A}
         \renewcommand{\theenumi}{A.\arabic{enumi}}
         \renewcommand{\labelenumi}{\theenumi}
         \begin{enumerate}
                 \item \label{as_A1} It holds $\mathbb{P}$-a.s.\ that for all $x\in\Z$ there exists $p_x\in (0,1)$ such that $\omega(x,i)=p_x$ for all $i> M_x$.
                 \item \label{as_A2} $(p_z)_{z\in\Z}$ is identically distributed and $\{p_z, M_z, z\in\Z\}$ is independent under $\mathbb{P}$.
		 \item \label{as_A3} $\mathbb{E}[|\log (\rho_0)|]<\infty$ and $\mathbb{E}[\log(\rho_0)]>0$ where $\rho_x:=({1-p_x})p_x^{-1}$ for $x\in\Z$.
		 \item \label{as_A4} $\mathbb{P}[M_x=\infty]=0$ and $\mathbb{P}[M_x=0]>0$ for all $x\in\Z$.
         \end{enumerate}
\end{assumptionletter} 

Recall that \ref{as_A3} implies that an RWRE is left-transient, i.e.\ $S_n\to -\infty$ a.s. The goal of this paper is to show the following recurrence and transience criteria.

\begin{theorem}
\label{theorem1}
Let Assumption \ref{as_A} 
hold and assume that
$(p_z, M_z)_{z\in\Z}$ is i.i.d.\ and $\mathbb{E}[p_0^{-1}]<\infty$.
\begin{itemize}
 \item[(i)] If $\mathbb{E}[(\log M_0)_+]<\infty$, then $\lim_{n\to\infty}S_n=-\infty$ $P_0$-a.s.
 \item[(ii)]  If $\mathbb{E}[(\log M_0)_+]=\infty$ and if $\limsup_{t\to\infty}(t\cdot\mathbb{P}[\log{M_0}>t])<\mathbb{E}[\log \rho_0]$, then $S_n=0$ infinitely often $P_0$-a.s.
 \item[(iii)] If
$\liminf_{t\to\infty}(t\cdot\mathbb{P}[\log{M_0}>t])>\mathbb{E}[\log \rho_0]$, then $\lim_{n\to\infty}S_n=+\infty$ $P_0$-a.s. 
\end{itemize}
\end{theorem}

In \cite{Basdevant08b, Basdevant08a, Zerner08, Zerner05, Zerner06} the random walker steps to the right or to the left with equal probability at sites without cookies. Furthermore, the number of cookies per site is bounded in \cite{Basdevant08b, Basdevant08a, Zerner08}.
In the case of ERWs in two and higher dimensions, Holmes investigates in \cite{Holmes11} the velocity, allowing infinitely many cookies, and discusses a question of ``which drift wins?'' for a certain subclass of models, so-called ``excited against the tide'' walks.
Schapira considers in his unpublished paper \cite{Schapira08} a model similar to the one used in this paper, where the cookies induce a positive drift and the walker gets a negative drift on sites without cookies. But the number of cookies per site is still bounded.
The novelty in our model is that the number of cookies per site is not necessarily bounded and that we allow a random environment for the transition probabilities on sites without cookies. However, cookies of maximal strength are considered only.

For the proof of Theorem \ref{theorem1} a well-known relationship between branching processes and random walks is used. This method also has been employed in \cite{Basdevant08b, Basdevant08a, Zerner08}. In our case we have to deal with a subcritical branching process in a random environment with immigration and  we also intend to prove a recurrence/transience criterion for this process. As can be seen in our further discussion there is some similarity to random difference equations.\\
Since the late 1960s several authors worked on branching processes in random environments (BPRE for short), for example Athreya, Karlin, Smith and Wilkinson in \cite{Athreya_BPRE2, Athreya_BPRE1, Smith69}.
In \cite{Pakes79} Pakes proved some recurrence/transience criteria for subcritical branching processes with immigration, but without a random environment. The results presented in our paper differ from the ones in \cite{Pakes79}, mainly because of the extension to a random environment, but also if we assumed fixed environments.

Let us now give the structure of this paper in a few words.
Section \ref{section:GWPI_in_RE} is dedicated to subcritical branching processes in random environments with immigration. In Section \ref{sec:pos_halfline} we consider the model with cookies on the positive integers only. The process with cookies on the negative integers only will be discussed in Section \ref{sec:neg_halfline}. Finally, Theorem \ref{theorem1} is proven in Section \ref{sec:ERWRE} and examples are given for the different cases of this theorem.

\section{Transience and recurrence of subcritical branching processes in a random environment with immigration}
\label{section:GWPI_in_RE}

For information about branching processes in general we refer to \cite{Athreya}.
In our paper we deal with the following Smith-Wilkinson model extended by immigration, see also \cite{Athreya_BPRE1} and \cite[VI.5 and VI.7.1]{Athreya}:
\begin{Definition}
\label{def:GWPI1}
Let $\Lambda$ denote the set of probability generating functions (p.g.f.) that is isomorphic to the set of probability distributions on $\N_0$ and let $\vartheta$ be a probability distribution on $\N_0$. Let $\{\xi_i^{(j)},M_n, i,j,n\in\N \}$ be a family of $\N_0$-valued random variables on a suitable probability space with sample space $\Lambda '$ equipped with a set of measures $\{Q_{\varphi}, \varphi\in\Lambda^{\N}\}$ such that for every fixed \emph{environment} $\varphi=(\varphi_j)_{j\in\N}\in\Lambda^{\N}$ the family $\{\xi_i^{(j)},M_n, i,j,n\in\N \}$ is independent, $(\xi_i^{(j)})_{i\in\N}$ is identically distributed with p.g.f.\ $\varphi_j$ for $j\in\N$ and $(M_n)_{n\in\N}$ is identically $\vartheta$-distributed under $Q_{\varphi}$. We call $(Z_n)_{n\geq 0}$, defined by $Z_0=1$ and
\begin{align*}
 Z_{n} &:= \xi_1^{(n)}+\ldots+\xi_{Z_{n-1}}^{(n)}+M_{n}
\end{align*}
a \emph{branching process with immigration}.\\
Let $\nu$ denote a probability measure on $\Lambda$ and $\varphi$ be randomly chosen according to $\tilde{Q}:=\otimes_{\N}\:\nu$ on $\Lambda^{\N}$ with corresponding expectation operator $\tilde{E}$. Now we define the annealed measure by $Q[\cdot]:=\tilde{E}[Q_{\varphi}[\cdot]]$, by $E_Q$ its expectation operator and call $(Z_n)_{n\geq 0}$ a \emph{BPRE with immigration} under $Q$.
\end{Definition}

In the above definition $\xi_i^{(n)}$ gives the number of progeny of the $i^{th}$ member of generation $n-1$, and $M_n$ the number of immigrants in generation $n$. All members in a generation reproduce according to the same offspring distribution. Due to our further discussion, especially in Section \ref{sec:pos_halfline}, we use the same notation for immigrants as for cookies.
Note that $((\varphi_n, M_n))_{n\in\N}$ is identically distributed and $\{\varphi_n, M_n, n\in\N\}$ is independent under $Q$. Furthermore, we remark that $(Z_n)_{n\geq0}$ is a time-homogeneous Markov chain under $Q$. In order to speak about recurrence and transience of the process it is assumed that the BPRE with immigration is irreducible under $Q$.

Let us set $E_{\varphi}$ for the expectation operator corresponding to $Q_{\varphi}$ and $\mu_n(\varphi):=E_{\varphi}[\xi_1^{(n)}]$ for the expected number of offspring produced by a single member of generation $n-1$. Analogously to the literature, we call $(Z_n)_{n\geq0}$ \emph{subcritical, critical or supercritical}, if ${E}_Q[\log \mu_1]<0, =$ or $>0$ respectively. This is due to the fact that under the first two assumptions, a BPRE without immigration is mortal, whereas under the third assumption this process can explode, see \cite[Theorem 3.1]{Smith69} or \cite[VI.\ 5.5]{Athreya}.
Note the analogy to Theorem RWRE $(i)-(iii)$, see also \cite[Remark 2]{Zerner08}.

Let us now give criteria for recurrence and transience of a subcritical BPRE with immigration in Theorems \ref{theorem:GWPI_in_RE} and \ref{theorem:GWPI_in_RE2}. We remark that these criteria are similar to the ones for random difference equations found in \cite[Theorem 3.1]{Kellerer}.

\begin{theorem}
\label{theorem:GWPI_in_RE}
Let $(Z_n)_{n\geq0}$ be a BPRE with immigration according to Definition \ref{def:GWPI1}.
Suppose that $E_Q[|\log (\mu_1)|]<\infty$, $E_Q[\log (\mu_1)]<0$, $E_Q[\mu_1^{-\theta}]<\infty$ for some $\theta>0$ and $E_Q[\Var_{\varphi}(\xi_1^{(1)})\cdot\mu_1^{-2}]<\infty$.\\
If $\liminf_{t\to\infty}(t\cdot Q[\log{M_1}>t])> E_Q[\log (\mu_1^{-1})]$, then $(Z_n)_{n\geq 0}$ is transient.
\end{theorem}

\begin{proof}
The proof will be divided into two parts. First, we discuss a random difference equation $X_n$ for $n\in\N_0$, and show that $Q[\bigcap_{n\geq 1}\{X_n>n^2\}]>0$.
The second part consists in coupling $(X_n)_{n\geq0}$ and the subcritical BPRE with immigration in order to show that $Q_{\varphi}\left[\lim_{n\to\infty}Z_n=\infty\right]>0$ for $Q$-a.e.\ $\varphi$.

First, let us define some constants. We write 
\[x:=E_Q[\log(\mu_1^{-1})]>0\]
and choose $c_1<\infty$, $c_2>1$ and $c_3>1$ such that
\[\liminf_{t\to\infty}(t\cdot Q[\log{M_1}>t])>c_1>\log c_2 +c_3 x>0.\]
Furthermore, let $\epsilon \in \left(0, c_1-\left(\log{c_2} + c_3 x\right)\right)$ and define
\[\gamma := \frac{c_1}{\epsilon+\log{c_2} + c_3 x}>1.\]
We choose $0<l<1$ such that $(1-l) \gamma >1$.

In order to make use of a convergence property, set
\[T:=\inf\left\{k\in\N: \forall n\geq k: \frac{1}{n}\sum_{m=1}^{n}\log \mu_m \geq - c_3 x\right\}.\]
Then, $T<\infty$ $Q$-a.s.\ by the law of large numbers.
A result about large deviations under the assumption $E_Q[\mu_1^{-\theta}]<\infty$, see \cite[p. 71]{Durrett}, gives
\[Q\left[\sum_{m=1}^{n}\log \mu_m^{-1} > n c_3 x\right]\leq e^{-n\cdot c}\]
for some $c>0$. This implies
\begin{equation}
\label{eq:large_deviations}
 E_Q\left[T^{\frac{1}{l}}\right]=\sum_{n\geq 1}Q[T=n]\cdot n^{\frac{1}{l}}\leq \sum_{n\geq 1}Q\left[\sum_{m=1}^{n-1}\log \mu_m^{-1} > (n-1) c_3 x\right]\cdot n^{\frac{1}{l}}<\infty.
\end{equation}

We begin now with a random difference equation defined by $X_0:= 0$ and
\begin{align*}
 X_{n+1} &:= \alpha_{n+1}X_n + M_{n+1}, 
\end{align*}
with $\alpha_n:={\mu_n}{c_2^{-1}}<\mu_n$. Induction yields
\begin{equation*}
 X_n =\alpha_n\ldots\alpha_2 M_1 + \alpha_n\ldots\alpha_3 M_2 +\ldots + \alpha_n M_{n-1}+ M_n.
\end{equation*}

We follow \cite{Kellerer} in setting
\[W_n:= M_1 + \alpha_1 M_2 +\alpha_1\alpha_2 M_{3}+\ldots + \alpha_1\ldots\alpha_{n-1} M_n\]
for $n\in\N$. Since $(\varphi_1,M_1),\ldots,(\varphi_n, M_n)$ are exchangeable, the law of $X_n$ and the law of $W_n$ are the same,
in particular
\begin{equation}
\label{eq:rde_1}
 Q[X_n\leq n^2]=Q[W_n\leq n^2] \quad\text{for all }n\in\N.
\end{equation}

Our first goal is to show that 
\begin{equation}
\label{eq:rde_2}
 Q\left[\liminf\limits_{n\to\infty}\{X_n>n^2\}\right]=1.
\end{equation}

Therefore, let us start with a discussion of $W_n$:
\begin{align}
 \label{eq:rde_3}
 Q[W_n\leq n^2]&=Q[W_n\leq n^2, T\geq n^l] + Q[W_n\leq n^2, T< n^l]\notag\\
&\leq Q[T\geq n^l] + Q[W_n\leq n^2, T< n^l],
\end{align}

and by definition of $\alpha_i$ and $T$:
\begin{align}
\label{eq:X_nkleiner_n2}
Q[W_n\leq n^2, T< n^l] 
&\leq Q\Bigg[\bigcap_{n^l\leq i < n}\left\{M_{i+1}\leq n^2\left(\alpha_1\ldots\alpha_{i}\right)^{-1}\right\},T< n^l \Bigg]\notag\\
&\leq Q\Bigg[\bigcap\limits_{n^l\leq i < n}\left\{M_{i+1}\leq n^2 (c_2\cdot e^{c_3 x})^{i}\right\},T< n^l \Bigg]\notag\\
&\leq Q\Bigg[\bigcap\limits_{n^l\leq i < n}\left\{M_{i+1}\leq n^2 (c_2\cdot e^{c_3 x})^{i}\right\} \Bigg]\notag\\
&= \prod\limits_{n^l\leq i < n}Q\left[M_{1}\leq n^2 (c_2\cdot e^{c_3 x})^{i} \right]\notag\\
&= \exp{\Bigg(\sum\limits_{n^l\leq i< n}\log\left(1-Q\left[M_{1}> n^2\left(c_2 e^{c_3 x}\right)^i\right]\right)\Bigg)}\notag \\
  &\leq  \exp{\Bigg(-\sum\limits_{n^l\leq i< n} Q\left[M_{1}> n^2\left(c_2 e^{c_3 x}\right)^i\right]\Bigg)}.
\end{align}

According to $\liminf_{t\to\infty}(t\cdot Q[\log{M_1}>t])> c_1$, we get for $n$ large enough
 \begin{align}\label{eq:Trans_Schritt3}
\sum\limits_{n^l\leq i< n}Q\left[M_{1}> n^2\left(c_2 e^{c_3 x}\right)^i\right]
&> \sum\limits_{n^l\leq i< n}{\frac{c_1}{2\log n + i\left(\log{c_2} +c_3 \cdot x\right)}} \notag\\
&\geq  \sum\limits_{n^l\leq i< n}{\frac{c_1}{i\epsilon + i\left(\log{c_2} +c_3 \cdot x\right)}}\notag\\
&= \sum\limits_{1\leq i< n}\gamma\cdot\frac{1}{i}
- \sum\limits_{1\leq i< n^l}\gamma\cdot\frac{1}{i} \notag\\
&\geq \gamma\left(\log(n)-\left(\log\left(n^l\right)+1\right)\right).
\end{align}

For the last inequality we used the convergence of $(\sum_{i=1}^{n}{\frac{1}{i}}-\log n)$ as $n\to\infty$ twice.
Thus, combining (\ref{eq:X_nkleiner_n2}) with (\ref{eq:Trans_Schritt3}) yields for $n$ large enough
\begin{align*}
Q[W_n\leq n^2, T< n^l] &\leq 
e^{-\gamma\cdot\log\left(n^{1-l}\right)}
\cdot e^{\gamma}
\leq n^{(l-1)\cdot\gamma}\cdot e^{\gamma}
\end{align*}

And finally we get by (\ref{eq:rde_1}) and (\ref{eq:rde_3})
\begin{align*}
\sum\limits_{n\geq1}Q\left[X_n\leq n^2\right]
\leq \sum\limits_{n\geq1}(Q[T^{\frac{1}{l}}\geq n]+Q[W_n\leq n^2, T< n^l])
<\infty 
\end{align*}
since $\sum_{n\geq1}n^{(l-1)\cdot\gamma}<\infty$ and $\sum_{n\geq1}Q[T^{\frac{1}{l}}\geq n]\leq E_Q[T^{\frac{1}{l}}]<\infty$, see (\ref{eq:large_deviations}).

Thus, by the Borel-Cantelli Lemma, we obtain (\ref{eq:rde_2})
and one can check that
\begin{equation}
\label{eq:immer_Xn>n2}
Q_{\varphi}\left[\bigcap_{n\geq 1}\{X_n>n^2\}\right]>0 \quad\text{for $Q$-a.e.\ }\varphi.
\end{equation}
Therefore, note that for $Q$-a.e.\ $\varphi$ there exist $\tilde{x}_0\geq x_0>n_0^2$ with 
\begin{equation*}
Q_{\varphi}\bigg[\{X_{n_0}=x_0\}\cap\bigcap\limits_{n\geq {n_0+1}}\{X_n>n^2\}\bigg]>0 \text{ and } Q_{\varphi}\bigg[\{X_{n_0}= \tilde{x}_0\}\cap\bigcap\limits_{j=1}^{n_0-1}\{X_j>j^2\}\bigg]>0.
\end{equation*}
Let us define recursively two processes
\begin{align*}
Y_0^{(x_0)}&:= x_0, &Y_n^{(x_0)}&:= \alpha_{n_0+n}\cdot Y_{n-1}^{(x_0)}+M_{n_0+n}\\
Y_0^{(\tilde{x}_0)}&:=\tilde{x}_0, &Y_n^{(\tilde{x}_0)}&:= \alpha_{n_0+n}\cdot Y_{n-1}^{(\tilde{x}_0)}+M_{n_0+n} \quad \text{for }n\in\N.
\end{align*}
Since $\tilde{x}_0\geq x_0$ we have that $Y_n^{(\tilde{x}_0)}\geq Y_{n}^{(x_0)}$. Thus, (\ref{eq:immer_Xn>n2}) follows by
\begin{align*}
Q_{\varphi}&\bigg[\bigcap\limits_{n\geq 1}\{X_n>n^2\}\bigg]\geq Q_{\varphi}\bigg[\{X_{n_0}= \tilde{x}_0\}\cap\bigcap\limits_{j=1}^{n_0-1}\{X_j>j^2\}\cap\bigcap\limits_{n\geq 1}\{Y_n^{(\tilde{x}_0)}>(n_0+n)^2\bigg]\\
&\geq Q_{\varphi}\bigg[\{X_{n_0}= \tilde{x}_0\}\cap\bigcap\limits_{j=1}^{n_0-1}\{X_j>j^2\}\bigg]\cdot Q_{\varphi}\bigg[\bigcap\limits_{n\geq 1}\{Y_n^{(x_0)}>(n_0+n)^2\bigg]\\
&>0.
\end{align*}

Our next objective is to couple $(X_n)_{n\geq 0}$ and BPRE with immigration. Therefore,
recall Definition \ref{def:GWPI1} for the branching process $(Z_n)_{n\geq0}$.
If we couple these processes, the increments of the difference equation correspond to the immigrants in the BPRE and the multiplication with $\alpha_n$ to the expected number of descents of one individual in generation $n-1$, multiplied by $c_2^{-1}$.
Our goal is to show that $Q_{\varphi}\left[\bigcap_{n\geq 1}\{Z_n\geq X_n > n^2\}\right]>0$ for $Q$-a.e.\ $\varphi$.

We define for $n\in\N_0$
\[B_n:=\bigcap_{j=1}^{n}\{Z_j\geq X_j\}\cap\bigcap_{k\geq 1}\{X_k >k^2\}.\]
Then, we get for $Q$-a.e.\ $\varphi$,
\begin{align*}
Q_{\varphi}&\left[Z_n < X_n, B_{n-1}\right]
=\sum\limits_{k> (n-1)^2}Q_{\varphi}\left[ Z_n < X_n,\;Z_{n-1}=k, B_{n-1}\right] \notag\\
&\leq\sum\limits_{k> (n-1)^2}Q_{\varphi}\left[ \mu_n\cdot k-\sum\limits_{i=1}^{k}{\xi_i^{(n)}}> \left(\mu_n-\alpha_n\right)\cdot k,\;Z_{n-1}=k, B_{n-1}\right] \notag\\
&=\sum\limits_{k> (n-1)^2}Q_{\varphi}\left[ \mu_n\cdot k-\sum\limits_{i=1}^{k}{\xi_i^{(n)}}> \left(1-c_2^{-1}\right)\mu_n\cdot k\right]\cdot Q_{\varphi}\left[Z_{n-1}=k, B_{n-1}\right].
\end{align*}

Now, Chebyshev inequality 
implies 
\begin{align*}
Q_{\varphi}\left[ \mu_n\cdot k-\sum\limits_{i=1}^{k}{\xi_i^{(n)}}> \left(1-c_2^{-1}\right)\mu_n\cdot k\right]
&\leq\frac{\Var_{\varphi}\left(\sum\limits_{i=1}^{k}{\xi_i^{(n)}}\right)}{\left(1-c_2^{-1}\right)^2\cdot \mu_n({\varphi})^2\cdot k^2}\\
&\leq  \frac{\Var_{\varphi}\left({\xi_1^{(n)}}\right)}{\left(1-c_2^{-1}\right)^2\cdot \mu_n({\varphi})^2\cdot k}\; .
\end{align*}

Thus, we have
\begin{align*}
 Q_{\varphi}\left[Z_n < X_n, B_{n-1}\right] < \frac{\Var_{\varphi}\left({\xi_1^{(n)}}\right)}{\left(1-c_2^{-1}\right)^2\cdot \mu_n({\varphi})^2\cdot (n-1)^2}\cdot Q_{\varphi}\left[B_{n-1}\right]
\end{align*}
and therefore
\begin{align}
\label{eq:sum_Qn}
 \sum_{n\in\N}&Q_{\varphi}\left[Z_n < X_n | B_{n-1}\right]<\infty \quad Q\text{-a.s.}
\end{align}
according to assumption $E_Q[\Var_{\varphi}(\xi_1^{(1)})\cdot\mu_1^{-2}]<\infty$ and \cite[Theorem 4.2.1.]{Lukacs75}.
It is easy to see that for all $n\geq 1$ and $Q$-a.e.\ $\varphi$ $\quad Q_{\varphi}\left[Z_n < X_n | B_{n-1}\right]<1 .$
Now we conclude from (\ref{eq:immer_Xn>n2}) and (\ref{eq:sum_Qn}) that
\begin{equation*}
Q_{\varphi}\left[\bigcap\limits_{n\in \N}\{Z_n\geq X_n\}\Bigg| \bigcap\limits_{k\in\N}\{X_k> k^2\}\right]
=\prod\limits_{n\in \N}{Q_{\varphi}\left[Z_n \geq X_n | B_{n-1} \right]}
> 0 \; .
\end{equation*}

Combining this 
with (\ref{eq:immer_Xn>n2}) we can assert that for $Q$-a.e.\ $\varphi$
\begin{align*}
Q_{\varphi}\left[\lim\limits_{n\to\infty}Z_n=
\infty\right]&\geq 
Q_{\varphi}\left[\bigcap\limits_{n\in\N}\{Z_n\geq X_n\}
\Bigg| \bigcap\limits_{k\in\N}\{X_k> k^2\}
\right] \cdot Q_{\varphi}\left[\bigcap\limits_{k\in\N}\{X_k> k^2\}
\right]
>0.
\end{align*}

Therefore, \[Q\left[\lim\limits_{n\to\infty}Z_n=\infty\right]=E_Q\left[Q_{\varphi}\left[\lim\limits_{n\to\infty}Z_n=\infty\right]\right]>0\]
which completes the proof.
\end{proof}

The following theorem gives a criterion for recurrence.

\begin{theorem}
\label{theorem:GWPI_in_RE2}
Let $(Z_n)_{n\geq0}$ be a BPRE with immigration according to Definition \ref{def:GWPI1} with $Q[M_n=0]>0$ and $\varphi_1(0)>0$ $Q$-a.s.
Suppose that $E_Q[|\log (\mu_1^{-1})|]<\infty$ and $E_Q[\log (\mu_1^{-1})]>0$.\\
If $\limsup_{t\to\infty}(t\cdot Q[\log{M_1}>t])< E_Q[\log (\mu_1^{-1})]$, then $(Z_n)_{n\geq 0}$ is recurrent.
\end{theorem}

\begin{proof}
We will show that $\sum_{n\geq0}Q[Z_n=0]=\infty$, which is equivalent to the recurrence of the process since $(Z_n)_{n\in\N_0}$ is a time-homogeneous Markov chain under $Q$. For the proof we use a second well-known definition for a BPRE with immigration, illustrated in Figure \ref{fig:BPRE}, and employ the exchangeability of $((\varphi_k,M_k))_{1\leq k\leq n}$ under $Q$ for $n\in\N$.

Equivalently to Definition \ref{def:GWPI1} the BPRE with immigration can also be defined by
\[Z_n =\sum\limits_{j=0}^{n}{Z_{n-j}(j)}, \quad n\in\mathbb{N}_0,\]
where for $j\in\N_0$, $(Z_n(j))_{n\in\N_0}$ is an ordinary BPRE characterized by offspring p.g.f.\ $(\varphi_{n+j})_{n\in\N}$ with $Z_0(0)=1$ and $Z_0(j)=M_j$ for $j\geq 1$. Furthermore, the BPREs have to be independent under conditioning on $\varphi$, see \cite[p.250]{Athreya}. For an explanation see Figure \ref{fig:BPRE}.

\begin{figure}[!ht]%
\centering

\scalebox{0.65}{\includegraphics[width=\columnwidth]{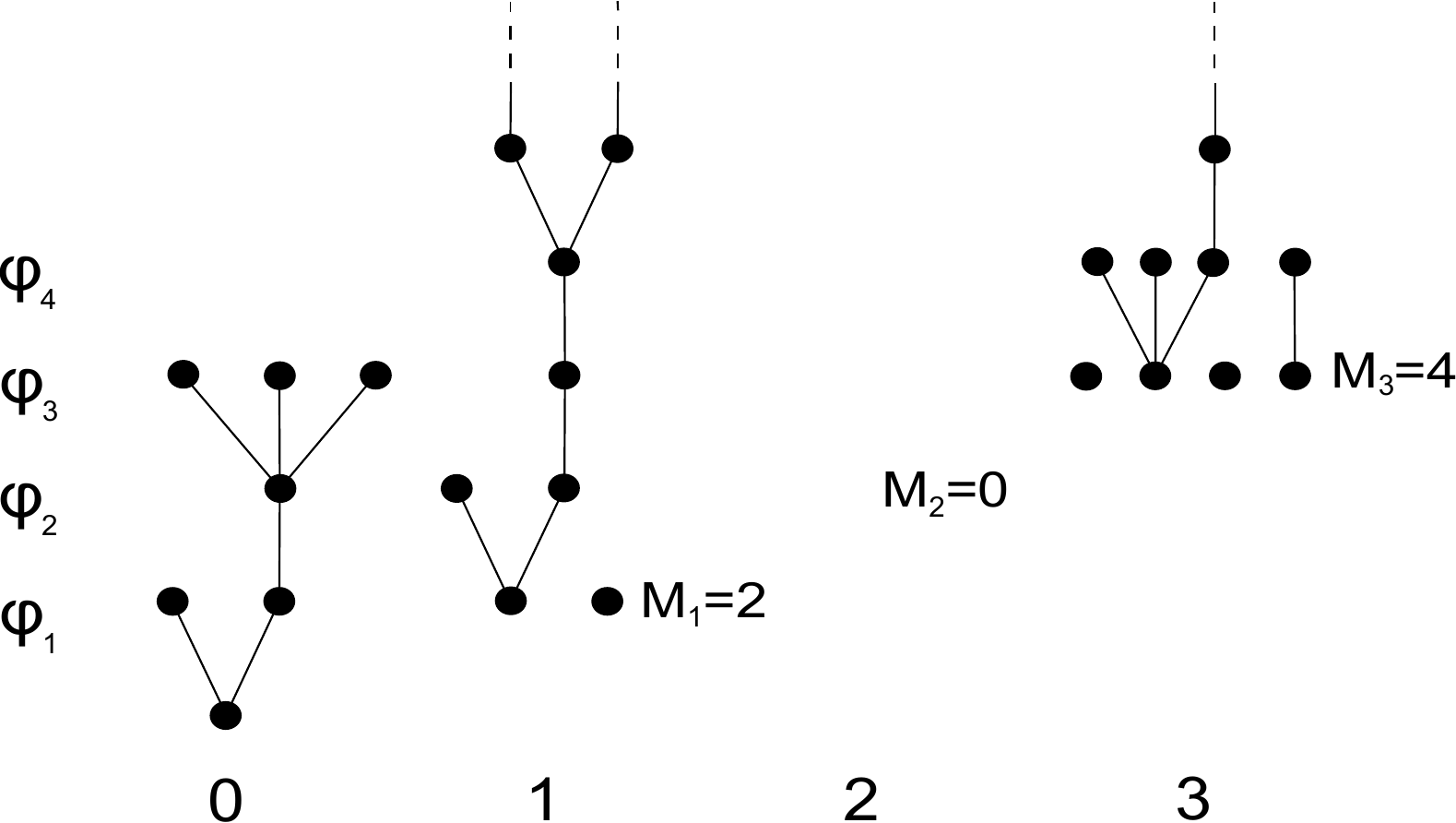}}
\caption{BPRE with immigration:}
{\parbox{0.85\textwidth}{\small The figure shows a realization of the first generations of the BPREs $(Z_n(j))_{n\in\N_0}$, $j\in\{0,\ldots,3\}$. The p.g.f.\ $\varphi_i$ for reproduction is written on the left side. Summation over the $n^{th}$ row gives $Z_n$.}}
\label{fig:BPRE}
\end{figure}

For the moment let us fix $n\in\N$. We exchange $(\varphi_1,M_1)$ by $(\varphi_n,M_n)$, $(\varphi_2,M_2)$ by $(\varphi_{n-1},M_{n-1})$, $\ldots$ $(\varphi_n,M_n)$ by $(\varphi_1,M_1)$
and define $Z_0'(0)=1$ and
\[Z_n' =\sum\limits_{j=0}^{n}{Z_{n-j}'(j)},\]
where $Z_{n-j}'(j)$ uses the exchanged random vectors $((\varphi_i,M_i))_{1\leq i\leq n}$, i.e.\ the distribution of $Z_{n-j}'(j)$ is given by $\varphi_{n-j}(\varphi_{n-j-1}(\ldots\varphi_{1}(s)\ldots))^{M_{n-j+1}}$ --- whereas the distribution of $Z_{n-j}(j)$ is given by $\varphi_{j+1}(\varphi_{j+2}(\ldots\varphi_{n}(s)\ldots))^{M_{j}}$ --- and $Z_0'(j)=M_{n-j+1}$.
Exchangeability of $((\varphi_k,M_k))_{1\leq k\leq n}$ thus implies
\[Q[Z_n=0]=Q[Z_n'=0] \qquad\text{ for all } n\in\N.\]

The task is now to show that 
\[\sum_{n\geq0}Q[Z_n'=0]=\infty.\]
Therefore let us examine $Q_{\varphi}[Z_n'=0]$. The strategy is the following:\\
In the first generations a high value of immigration is allowed --- knowing that the size of the population becomes small, when the time passes by, since we consider a subcritical branching process --- whereas later, when we come closer to the generation to be examined, there are only few people permitted to immigrate.

Choose $c>0$ such that $\limsup_{t\to\infty}(t\cdot Q[\log M_1>t])<c<E_Q[\log (\mu_1^{-1})]$.
Thus, we obtain for large $t\in\R$,
\begin{equation}
   Q[M_1>t]<\frac{c}{\log t}<1.
  \label{eq:vb1}
  \end{equation}

We consider some $\gamma>1$ such that $\gamma\cdot c<E_Q[\log (\mu_1^{-1})]$ and define for $k\in\mathbb{N}$
\begin{equation*}
c_k:= e^{\gamma c k}>1.
\label{eq:c_n}
\end{equation*}

Note that the lines of descent are independent under conditioning on $\varphi$ 
and that
\begin{align*}
 Q_{\varphi} \left[Z_{j}'(n-j)=0, Z_0'(n-j)\leq c_{j}\right] 
\geq \varphi_{j}(\varphi_{j-1}(\ldots\varphi_{1}(0)\ldots))^{c_{j}}\cdot Q_{\varphi} \left[Z_0'(n-j)\leq c_{j}\right]
\end{align*}
for $1\leq j\leq n$.
Hence, we get for every $n\in\N$,
\begin{align}
\label{eq:def_a_n}
Q_{\varphi}\left[Z_n'=0\right]
&\geq Q_{\varphi}\left[\sum\limits_{j=0}^{n}{Z_{n-j}'(j)=0},\; \bigcap\limits_{j=0}^{n-1}\{Z_0'(j)\leq c_{n-j}\}\right]\notag\\
&= Q_{\varphi}[Z_0'(n)=0]\cdot \prod\limits_{j=0}^{n-1}{Q_{\varphi} \left[Z_{n-j}'(j)=0, Z_0'(j)\leq c_{n-j}\right]}\notag\\
&\geq Q[M_1=0]\cdot \prod\limits_{j=1}^{n}\varphi_{j}(\varphi_{j-1}(\ldots\varphi_{1}(0)\ldots))^{c_{j}}\cdot \prod\limits_{j=1}^{n-1} Q \left[M_{j+1}\leq c_{j}\right]\notag\\
&=:a_n.
\end{align}

Since $a_n>0$ $Q$-a.s., we are able to consider for $n\geq 2$
\begin{align}
\label{eq:frac_a_n}
 \frac{a_n}{a_{n-1}}
&= \varphi_{n}(\varphi_{n-1}(\ldots\varphi_{1}(0)\ldots))^{c_{n}} \cdot Q_{\varphi} \left[M_{n}\leq c_{n-1}\right].
\end{align}

In the case of subcritical BPREs, convexity implies, see e.g. \cite[p.1853]{Athreya_BPRE2},
\[\varphi_{n}(\varphi_{n-1}(\ldots \varphi_{1}(0)))\geq 1-\mu_{1}\cdot\ldots\cdot \mu_{n}\;.\]

Choose $\epsilon >0$ such that $\gamma\cdot c+ E_Q[\log \mu_1]+\epsilon <0$ and let 
\[T:=\inf\{k\in\N: \forall n\geq k: \frac{1}{n}\sum_{1\leq m\leq n}\log \mu_m \leq E_Q[\log \mu_1]+\epsilon\}.\]

By the law of large numbers, $T<\infty$ $Q$-a.s.
For every $N\in\N$ we obtain on $\{T\leq N\}$ for $n\geq N$,
\begin{align}
\label{eq:phi_n__phi_0}
 \varphi_{n}(\varphi_{n-1}(\ldots\varphi_{1}(0)\ldots))^{c_{n}}
&\geq (1-\mu_n \cdot\ldots\cdot \mu_1)^{c_n}\notag\\
&\geq (1-e^{ n (E_Q[\log \mu_1]+\epsilon)})^{c_n}\notag\\
&\geq 1-e^{n(\gamma\cdot c+E_Q[\log \mu_1]+\epsilon)}.
\end{align}

Applying (\ref{eq:vb1}) yields
\begin{equation}
\label{eq:M_n<=c_n}
 Q[M_n\leq c_{n-1}]\geq 1-\frac{1}{\gamma (n-1)} \qquad\text{for $n$ large enough}.
\end{equation}

Thus, we get from (\ref{eq:frac_a_n}), (\ref{eq:phi_n__phi_0}) and (\ref{eq:M_n<=c_n}) $\frac{a_n}{a_{n-1}}\geq 1-\frac{1}{n-1}$ for $n$ large enough
and finally, by a criterion of Raabe and (\ref{eq:def_a_n}),
\[\sum_{n\geq0}Q_{\varphi}[Z_n'=0]=\infty\qquad \text{on }\{T\leq N\}\]
for any $N\in\N$ and hence $Q$-a.s. Therefore, we can conclude
\[\sum_{n\geq0}Q[Z_n=0]=\sum_{n\geq0}Q[Z_n'=0]=\sum_{n\geq0}E_Q[Q_{\varphi}[Z_n'=0]]=\infty.\]
\end{proof}

\section{Random walk in a random environment with cookies on the positive integers}
\label{sec:pos_halfline}
We now return to the random walk $(S_n)_{n\geq0}$.
This section deals with the case where cookies are only allowed on the positive integers. For its discussion a well-known connection between branching processes with migration and excited random walks is used. This idea was also employed in \cite{Basdevant08b, Basdevant08a, Zerner08} for a simple symmetric random walk disturbed by cookies. For detailed explanations we refer to \cite{Zerner08}.

In our case, an analogous relation between a subcritical BPRE with immigration and the ERWRE is obtained.
The purpose of this connection is to prove a recurrence and transience criterion for $(S_n)_{n\geq0}$.

\begin{proposition}
\label{theorem:ERWRE_right}
Let Assumption \ref{as_A} 
hold.
Assume $\mathbb{P}[M_z=0]=1$ for all $z\in -\N_0$, $(p_n, M_n)_{n\in\N}$ i.i.d.\ and $\mathbb{E}[p_0^{-1}]<\infty$.
           \begin{itemize}
                 \item[(i)] If $\liminf_{t\to\infty}(t\cdot\mathbb{P}[\log{M_1}>t])>\mathbb{E}[\log \rho_0]$,
                 then $P_0[\lim_{n\to\infty}S_n= +\infty]>0$.
                 \item[(ii)] If $\limsup_{t\to\infty}(t\cdot\mathbb{P}[\log{M_1}>t])<\mathbb{E}[\log \rho_0]$,
                 then $P_0[\lim_{n\to\infty}S_n= +\infty]=0$.
           \end{itemize}
\end{proposition}

Let us first derive the connection to BPRE with immigration and state some useful results.
We denote the hitting time of $k\in\Z$ by
\[T_k:=\inf\{n\in\N: S_n=k\}.\]

\begin{lemma}
\label{lem:Sn+-unendlich}
Let Assumption \ref{as_A} 
hold. Assume
$\mathbb{P}[M_z=0]=1$ for all $z\in -\N_0$ and $(p_n, M_n)_{n\in\N}$ i.i.d.
Then $P_0[\lim_{n\to\infty}S_n \in\{\pm\infty\}]=1$.
\end{lemma}

\begin{proof}
We first show that $P_{0,\omega}\left[-\infty <\limsup_{n\to\infty} S_n <+\infty\right]=0$ for $\mathbb{P}$-a.e.\ $\omega\in\Omega$ by using the strong Markov property for the process $H_n := (S_m)_{0\leq m\leq n}$ and Borel-Cantelli.
Assume that there exists $z\in\Z$ with $\limsup_{n\to\infty}S_n=z$. Then $z$ is visited infinitely many times. But so is $z+1$ since $p_z(\omega)>0$. This is a contradiction.

Similar considerations show that $P_{0,\omega}\left[-\infty <\liminf_{n\to\infty} S_n <+\infty\right]=0$ for $\mathbb{P}$-a.e.\ $\omega$.
If $\liminf_{n\to\infty} S_n=z$ for some $z\in\Z$, then $z$ will be visited infinitely many times. After the $M_z(\omega)^{th}$ visit, all cookies on $z$ have been removed and from this moment on, $S_n$ will visit $z-1$ infinitely many times as well, our next contradiction.

Finally, we claim that $P_{0,\omega}\left[\liminf_{n\to\infty} S_n=-\infty, \limsup_{n\to\infty} S_n =+\infty\right]=0$ for\linebreak $\mathbb{P}$-a.e.\ $\omega$.
This holds since --- under the assumptions of the lemma --- the environments on sites $z\leq 0$ satisfy case $(ii)$ of Theorem RWRE.
Hence, we have $P_{0,\omega}[T_1=\infty]>0$.
\end{proof}
A relation between an RWRE and a BPRE without immigration has already been remarked in \cite[Remark 2]{Zerner08}.
We show that our model of an ERWRE can be associated with a BPRE with immigration.
Consider $k\in\N$. Then, it holds for $\mathbb{P}$-a.e.\ $\omega\in\Omega$ that on the event $\{T_k<T_0<\infty\}$ the random walk has to do $M_k$ steps from $k$ to $k+1$ after its first visit in $k$ until all cookies are removed and until it gets a chance to return to $k-1$. Further steps from $k$ to $k+1$ can be regarded as the number of successes in coin tossing prior to the first failure with an unfair coin of probability $p_k$ to succeed.

Let $X_i^{(j)}$, $i\in\N$, $j\in\Z$, be a family of independent $\pm1$-valued random variables on $\Omega'$,
such that
\[P_{z,\omega}[X_i^{(j)}=1]=\omega(j,i)\text{  and  }P_{z,\omega}[X_i^{(j)}=-1]=1-\omega(j,i).\]
Following \cite{Zerner08}, the ERWRE can be realized recursively by
\[S_{n+1}=S_n + X_{\#\{m\leq n:\: S_m=S_n\}}^{(S_n)} \qquad\text{ for } n\geq0.\]
We call the events $\{X_i^{(j)}=1\}$ a success and $\{X_i^{(j)}=-1\}$ a failure and define
\begin{align*}
\xi_j^{(k)}&:= \# \{\text{successes in $\big(X_i^{(k)}\big)_{i>M_k}$ between the $(j-1)^{th}$ and the $j^{th}$ failure}\}\; ,\\
V_0&:=1\; ,\\
V_k&:=\xi_1^{(k)}+\ldots+\xi_{V_{k-1}}^{(k)}+M_k\; .
\end{align*}

Then, if Assumption \ref{as_A} and $(p_n, M_n)_{n\in\N}$ i.i.d.\ hold, we obtain by Definition \ref{def:GWPI1} that $(V_k)_{k\geq 0}$ is a BPRE with immigration under $P_1$, with immigrants $(M_k)_{k\geq1}$ and progeny given by $(\xi_i^{(j)})_{i,j\in\N}$, where $\xi_i^{(j)}$ has geometric distribution $geo_{\N_0}(1-p_j)$, i.e.\ $P_{1,\omega}[\xi_i^{(j)}=n]=p_j(\omega)^n\cdot (1-p_j(\omega))$ for $\mathbb{P}$-a.e.\ $\omega$.

For the formalization of the connection let us define analogously to \cite{Zerner08}
\[U_0 := 1 \text{  and  } U_k := \#\{n\geq 0: n< T_0, S_n =k, S_{n+1}=k+1\},\]
where $U_k$ denotes the number of upcrossings from $k$ to $k+1$ before the return to 0 for $k\in\N$.
Analogously to \cite[eq.\ (14),(15)]{Zerner08} the following connection between $U_k$ and $V_k$ can be formulated and proven.

\begin{lemma}
\label{lem:Relation-U-V}
Let Assumption \ref{as_A} hold and assume $(p_n, M_n)_{n\in\N}$ i.i.d.
It holds a.s.\ under $P_1$:
\begin{align*}
 &U_k=V_k\cdot\mathbbmss{1}_{\{V_j>0\; \forall 0\leq j<k\}} &\text{for all } k\geq1 \text{ on }\{T_0<\infty\} ,\\
&U_k\leq V_k
&\text{for all } k\geq1 \text{ on }\{T_0=\infty\} .
\end{align*}
\end{lemma}
The indicator function ensures the equality after the first return to 0 of the random walker or, equivalently, the first time the population became extinct.

Let us use the same definition of recurrence from the right as in \cite[p.1962]{Zerner08}.
\begin{Definition}
\label{def:rekurrentvonrechts}
The ERWRE is called \emph{recurrent from the right}, if the first excursion to the right of 0, if there is any, is $P_0$-a.s.\ finite, i.e.\ $P_1[T_0 <\infty]=1$.
\end{Definition} 

Now the connection between ERWRE and BPRE with immigration can be formulated in the following way.
\begin{lemma}
\label{lem:ZshRekurrenzERWundGWPI}
Let Assumption \ref{as_A} hold and assume $(p_n, M_n)_{n\in\N}$ i.i.d.
The ERWRE $(S_n)_{n\geq0}$ is recurrent from the right if and only if $(V_k)_{k\geq0}$ is recurrent, i.e.\ $P_1[\exists k\in\N: V_k=0]=1$.
\end{lemma}

\begin{proof}
The definition of $U_k$, Lemma \ref{lem:Sn+-unendlich} and Lemma \ref{lem:Relation-U-V} yield
\begin{equation}
\{T_0=\infty\}\overset{P_{1}}{=}\{U_k>0\;\forall k\geq1\} \overset{P_1}{\subseteq} \{V_k>0 \; \forall k\geq1\}.
\label{eq:Tunendlich_Upositiv}
\end{equation}
On the other hand we get by Lemma \ref{lem:Relation-U-V} and (\ref{eq:Tunendlich_Upositiv}) that
\[ P_1[V_k>0 \;\forall k\geq1, T_0<\infty]=P_1[U_k>0 \;\forall k\geq1, T_0<\infty]=0\]
and hence
\[\{V_k>0 \; \forall k\geq1\}\overset{P_1}{\subseteq} \{T_0=\infty\}.\]
Now, the lemma follows since
\begin{equation*}
\label{eq:1}
 P_1[T_0=\infty]=P_1[V_k>0 \;\forall k\geq1]=1-P_1[\exists k\in\mathbb{N}:\; V_k=0].
\end{equation*}
\end{proof}

\begin{lemma}
\label{lem:finite_excursion}
Let Assumption \ref{as_A} hold. Assume $\mathbb{P}[M_z=0]=1$ for all $z\in -\N_0$ and $(p_n, M_n)_{n\in\N}$ i.i.d.
 If $(S_n)_{n\geq0}$ is recurrent from the right, then all excursions are $P_0$-a.s.\ finite and $P_0$-a.s.\ there are only finitely many.
\end{lemma}

\begin{proof}
The first statement can be proven equivalently to the one in \cite[Lemma 8]{Zerner08}.
For the second statement assume that there exist infinitely many excursions to the right with positive probability. Then $S_n$ visits site 1 infinitely many times with positive probability. This is a contradiction to Lemma \ref{lem:Sn+-unendlich}.
\end{proof}

\begin{lemma}
\label{lem:Divergenz}
Let Assumption \ref{as_A} hold. Assume $\mathbb{P}[M_z=0]=1$ for all $z\in -\N_0$ and $(p_n, M_n)_{n\in\N}$ i.i.d.
 If $(S_n)_{n\geq0}$ is recurrent from the right, then $P_0\left[\lim_{n\to\infty}S_n=-\infty\right]=1$. If $(S_n)_{n\geq0}$ is not recurrent from the right, then $P_0\left[\lim_{n\to\infty}S_n=+\infty\right]>0$.
\end{lemma}

\begin{proof}
If the ERWRE is recurrent from the right, then according to Lemma \ref{lem:finite_excursion} all excursions are $P_0$-a.s.\ finite and there are only finitely many.
Therefore, we get $P_0\left[\lim_{n\to\infty}S_n= +\infty\right]=0$ and, applying Lemma \ref{lem:Sn+-unendlich}, $P_0\left[\lim_{n\to\infty}S_n= -\infty\right]=1$.

If $(S_n)_{n\geq0}$ is not recurrent from the right, the corresponding BPRE with immigration is transient by Lemma \ref{lem:ZshRekurrenzERWundGWPI}, i.e.\ $P_1[\exists k\in\mathbb{N}:V_k=0]<1$, and hence
\[0<P_1\left[V_k>0 \;\forall k\geq1\right]=P_1\left[T_0=\infty\right]\leq P_1\left[\lim\limits_{n\to\infty}S_n=+\infty\right].\]
The lemma follows.
\end{proof}

Lemmata \ref{lem:Divergenz} and \ref{lem:ZshRekurrenzERWundGWPI} finally show

\begin{proposition}
\label{theorem:combinationGWPI_and_ERW}
Let Assumption \ref{as_A} hold. Assume $\mathbb{P}[M_z=0]=1$ for all $z\in -\N_0$ and $(p_n, M_n)_{n\in\N}$ i.i.d.
 If $(V_k)_{k\geq0}$ is recurrent, then $P_0\left[\lim_{n\to\infty}S_n=-\infty\right]=1$. If $(V_k)_{k\geq0}$ is transient, then $P_0\left[\lim_{n\to\infty}S_n=+\infty\right]>0$.
\end{proposition}

The above results enable us to prove Proposition \ref{theorem:ERWRE_right}:

\begin{proof}[Proof of Proposition \ref{theorem:ERWRE_right}]
The BPRE with immigration $(V_k)_{k\geq 0}$ that corresponds to $(S_n)_{n\geq0}$ is given by immigrants $(M_n)_{n\geq1}$ and offspring distribution $geo_{\N_0}(1-p_j)$, $j\in\N$. Therefore, given an environment $\omega\in\Omega$,
the expected number of offspring produced by a single particle in the $(j-1)^{th}$ generation is
\begin{equation}
\label{eq:mu}
 \mu_j(\omega):=E_{0,\omega}[\xi_1^{(j)}]=\frac{p_j(\omega)}{1-p_j(\omega)}=\rho_j^{-1}(\omega)
\end{equation}
and
\begin{equation}
\label{eq:var}
 \Var_{\omega}(\xi_1^{(j)})=\frac{p_j(\omega)}{(1-p_j(\omega))^2}\; .
\end{equation}

In order to apply Theorems \ref{theorem:GWPI_in_RE} and \ref{theorem:GWPI_in_RE2} we first check if the particular assumptions are satisfied. We have $P_0=Q$. Assumption \ref{as_A} includes by ($\ref{eq:mu}$) that $E_Q[|\log (\mu_1)|]<\infty$, $E_Q[\log (\mu_1)]<0$, $P_0[M_1=0]>0$ and $P_{0,\omega}[\xi_1^{(1)}=0]=1-p_1(\omega)>0$ $P_0$-a.s. Since $\mathbb{E}[p_0^{-1}]<\infty$ we obtain that $E_0[\mu_1^{-1}]\leq E_0[p_1^{-1}]<\infty$ and $E_0[\Var_{\omega}(\xi_1^{(1)})\cdot \mu_1(\omega)^{-2}]=E_0[p_1^{-1}]<\infty$ by (\ref{eq:var}). Hence, the proof is completed by combining
Theorems \ref{theorem:GWPI_in_RE}, \ref{theorem:GWPI_in_RE2} and Proposition \ref{theorem:combinationGWPI_and_ERW}.
\end{proof}

\section{Random walk in a random environment with cookies on the negative integers}
\label{sec:neg_halfline}

In this section we put cookies of maximal strength on the negative integers only and prove the following proposition for the ERWRE.
\begin{proposition}
\label{RWRE_ERW_left}
Let Assumption \ref{as_A} 
hold and assume $\mathbb{P}[M_n=0]=1$ for all $n\in \N_0$
and $(p_z, M_z)_{z\in -\N}$ i.i.d.
\begin{itemize}
\item[(i)]  If $\mathbb{E}[(\log M_{-1})_+]<\infty$, then $\lim_{n\to\infty}S_n= -\infty$ $P_0$-a.s.
\item[(ii)]  If $\mathbb{E}[(\log M_{-1})_+]=\infty$, then $\liminf_{n\to\infty}S_n= -\infty$ and $\limsup_{n\to\infty}S_n= +\infty$ $P_0$-a.s.
\end{itemize}
\end{proposition}

The process can be illustrated by running against a wall and being thrown back. The random walk starts in 0 and, according to the environment, is pushed more or less to the left side, i.e.\ towards the negative integers, see Theorem RWRE $(ii)$.
But on this side it encounters sites with cookies, that push it back to the right, like a wall that it can't pass. Every time it returns to this site --- it will return, since a left-transient environment is considered --- the wall becomes smaller until it is all gone and the process can jump further left.
Naturally, the question arises how tall the walls have to be in order to make the random walker return to 0 between spotting and destroying a wall and in which cases this happens infinitely many times.
First, let us state that each negative integer, or each wall, will be reached by $(S_n)_{n\geq 0}$.
\begin{lemma}
\label{lemma3*}
Let Assumption \ref{as_A} 
hold. Assume $\mathbb{P}[M_n=0]=1$ for all $n\in \N_0$ and
$(p_z, M_z)_{z\in -\N}$ i.i.d.
Then $T_k<\infty$ $P_{0}$-a.s.\ for all $k\in-\N$.
\end{lemma}

\begin{proof}
If all cookies are removed we get a left-transient RWRE, i.e.\ an RWRE that tends to $-\infty$ $P_{0}$-a.s., by Theorem RWRE $(ii)$.
Now the lemma is proven by induction.
For $k=-1$ the statement is the same as for a left-transient RWRE. Therefore, it is true.
Let it hold for $k\in -\N$. Then, we get for $\mathbb{P}$-a.e.\ $\omega$,
\begin{align*}
 P_{0,\omega}[T_{k-1}<\infty]
= P_{0,\omega}[T_k<\infty,\inf\{n\in\N: S_{T_k +n}=k-1\}<\infty] 
= 1
\end{align*}
since $(S_{T_k +n})_{n\geq 0}$ acts like an RWRE on sites larger than $k$. Therefore, it returns to $k$ at least $M_k(\omega)$-times and reaches $k-1$ afterwards as well.
\end{proof}

Let $z\in\Z$. We define for $m\in\N$ with $-m<z$
\[A_m(z):=\{\text{after the first visit in }-m \quad (S_n)_{n\geq0} \text{ reaches }-m-1\text{ before }z\}.\]
Note that there are a.s.\ no cookies on the positive integers and by the first visit in $-m$ all cookies on sites larger than $-m$ have been removed. Since $\mathbb{E}[\log \rho_0]>0$ and by Theorem RWRE $(ii)$, 
the random walker returns to $-m$ from every excursion to the right after his first visit in $-m$ $P_{0,\omega}$-a.s.\ for $\mathbb{P}$-a.e.\ $\omega$. Therefore, 
\begin{equation}
\label{eq:limsupA_n}
 \{S_n=z \text{ infinitely often}\}^c\overset{P_0}{=}\{S_n<z \text{ eventually}\}\overset{P_0}{=}\bigcup_{k> (0\vee -z)}\bigcap_{m\geq k}A_m(z) ,
\end{equation}
where $\{\cdot\}^c$ denotes the complement of set $\{\cdot\}$.

By the strong Markov property $(A_n(z))_{n\in\N}$ is independent under $P_{0,\omega}$ for $\mathbb{P}$-a.e.\ $\omega\in\Omega$ and thus, by (\ref{eq:limsupA_n}) and Borel-Cantelli,
\begin{align}
\label{eq_Borel_Cantelli}
 P_{0,\omega}[S_n=z \text{ infinitely often}]
  &= \begin{cases}
  0 & \text{if } \sum_{n> (0\vee -z)}{P_{0,\omega}[A_n^c(z)]}<\infty \\
  1 & \text{if } \sum_{n> (0\vee -z)}{P_{0,\omega}[A_n^c(z)]}=\infty
      \end{cases}\; .
\end{align}

To get some criteria for convergence or divergence of this sum, we first 
have a closer look at an ordinary RWRE. The next lemma can be deduced from \cite[eq.\ (2.1.4)]{Zeitouni}.

\begin{lemma}
\label{lemma:Tk<T0}
Let $z\in\Z$, $k\in\N$ with $-k<z$.
The quenched probability for an RWRE with start in $-k$ to hit $-k-1$ before $z$ is 
\[1-\frac{1}{1+\sum_{l=-z+1}^{k}\prod_{j=l}^{k}\rho_{-j}}\; .\]
\end{lemma}

This result will be employed to prove the following lemma.
\begin{lemma}
\label{lemma2}
Let Assumption \ref{as_A} 
hold. Assume $\mathbb{P}[M_n=0]=1$ for all $n\in \N_0$ and $(p_z, M_z)_{z\in -\N}$ i.i.d.
For $z\in\Z$ and for $\mathbb{P}$-a.e.\ $\omega\in\Omega$ the following holds:
 \begin{itemize}
  \item[(i)]  If $\sum_{n>(0\vee -z)} P_{0,\omega}[A_n^c(z)]=\infty$, then 
 $\sum_{n\in\N}{M_{-n}(\omega)}{\prod_{j=1}^{n-1}\rho_{-j}^{-1}(\omega)}=\infty$.
  \item[(ii)] If $\sum_{n\in\N}{M_{-n}(\omega)}e^{-Cn}=\infty$ for some $C>\mathbb{E}[\log\rho_0]>0$, then $\sum_{n>(0\vee -z)}P_{0,\omega}[A_n^c(z)]=\infty$.
 \end{itemize}
\end{lemma}

\begin{proof}
Let $z\in\Z$. Lemma \ref{lemma:Tk<T0} implies for $n> (0\vee -z)$
\begin{align}
\label{lemma:P_An}
 P_{0,\omega}[A_n(z)]
  &= \left(1-\frac{1}{1+\sum_{l=-z+1}^{n-1}\prod_{j=l}^{n-1}\rho_{-j}}\right)^{M_{-n}} \left(1-\frac{1}{1+\sum_{l=-z+1}^{n}\prod_{j=l}^{n}\rho_{-j}}\right).
\end{align}

According to the law of large numbers, we can state for some $\epsilon>0$ and $\mathbb{P}$-a.e.\ $\omega$ that for $n$ large enough
\begin{equation}
\label{eq:2}
 \frac{1}{\prod_{j=-z+1}^{n}{\rho_{-j}}}=e^{\left(-\sum_{j=-z+1}^{n}\log \rho_{-j}\right)}\leq e^{-\epsilon (n+z)}
\end{equation}
and consequently
\begin{equation}
\label{eq:convergence}
 \sum_{n> (0\vee -z)}\frac{1}{\prod_{j=-z+1}^{n}{\rho_{-j}}}<\infty.
\end{equation}

Let us fix such an environment $\omega\in\Omega$ and let $n_0(\omega)\in\N$ with $n_0(\omega)> -z$ such that (\ref{eq:2}) holds for all $n\geq n_0(\omega)$. Then, we have for $n\geq n_0(\omega)$ by (\ref{lemma:P_An}) that
\begin{align*}
 P_{0,\omega}[A_n(z)] &\geq \left(1-\frac{1}{\prod_{j=-z+1}^{n-1}{\rho_{-j}}}\right)^{M_{-n}}\cdot \left(1-\frac{1}{\prod_{j=-z+1}^{n}{\rho_{-j}}}\right)  \\
  &\geq \left(1-\frac{{M_{-n}}}{\prod_{j=-z+1}^{n-1}{\rho_{-j}}}\right)\cdot \left(1-\frac{1}{\prod_{j=-z+1}^{n}{\rho_{-j}}}\right)
\end{align*}
and therefore,
\begin{align*}
  \sum_{n\geq n_0(\omega)} P_{0,\omega}[A_n^c(z)]&\leq\sum_{n\geq n_0(\omega)}{\left(1-\left(1-\frac{{M_{-n}}}{\prod_{j=-z+1}^{n-1}{\rho_{-j}}}\right)\cdot \left(1-\frac{1}{\prod_{j=-z+1}^{n}{\rho_{-j}}}\right)\right)}  \\
  &= \sum_{n\geq n_0(\omega)}{\left(\frac{1}{\prod_{j=-z+1}^{n}{\rho_{-j}}}+\frac{{M_{-n}}}{\prod_{j=-z+1}^{n-1}{\rho_{-j}}}\left(1-\frac{1}{\prod_{j=-z+1}^{n}{\rho_{-j}}}\right)\right)} \\
  &\leq \sum_{n\geq n_0(\omega)}{\left(\frac{1}{\prod_{j=-z+1}^{n}{\rho_{-j}}}+\frac{{M_{-n}}}{\prod_{j=-z+1}^{n-1}{\rho_{-j}}}\right)}.
\end{align*}
Hence, the first statement of the lemma follows by (\ref{eq:convergence}).

For the second part let $C>\mathbb{E}[\log\rho_0]>0$. Due to (\ref{lemma:P_An}), we get for $n\in\N$, $n> -z$
\begin{equation*}
 P_{0,\omega}[A_n(z)] \leq 
\left(1-\frac{1}{\rho_z\cdot\ldots\cdot\rho_{-n+1}\cdot\sum_{j=-z}^{n-1}{{\rho_z^{-1}\cdot\ldots\cdot\rho_{-j}^{-1}}}}\right)^{M_{-n}}.
\end{equation*}

By (\ref{eq:convergence}) and by applying the law of large numbers, there are $\mathbb{P}$-a.s.\ constants $c,\tilde{c}>0$ such that for large $n$,
\begin{align*}
 P_{0,\omega}[A_n(z)]  &\leq \big(1-e^{(-\sum_{j=-z}^{n-1}{\log \rho_{-j}})}\cdot c\big)^{M_{-n}} \\
 &\leq \left(1-{c\cdot e^{-C (n+z)}}\right)^{M_{-n}} \\
  &= \left(\left(1-{c\cdot e^{-z}\cdot e^{-Cn}}\right)^{e^{Cn}}\right)^{{M_{-n}}{e^{-Cn}}} \\
  &\leq e^{-\tilde{c}{M_{-n}} {e^{-Cn}}}.
\end{align*}

The second assertion of the lemma is obtained by the equivalence of
\[\prod\limits_{n\in\N}{a_n}>0 \quad\Leftrightarrow\quad \sum\limits_{n\in\N}{(1-a_n)}<\infty,\]
if $0< a_n \leq 1$ for all $n\geq1$.
\end{proof}

In order to prove Proposition \ref{RWRE_ERW_left}, we have to deal with random power series.
An application of \cite[Theorem 5.4.1]{Lukacs75} yields for our case the following result.
\begin{lemma}
\label{lemma_powerseries}
If $\mathbb{E}[(\log M_{-1})_+]=\infty$, then $\sum_{n \geq 1}{M_{-n}{x^n}}=\infty$ $\mathbb{P}$-a.s.\ for every ${x}>0$.\\
If $\mathbb{E}[(\log M_{-1})_+]<\infty$, then $\sum_{n \geq 1}{M_{-n}{x^n}}<\infty$ $\mathbb{P}$-a.s.\ for $0<{x}<1$.
\end{lemma}

Gathering the results above we get Proposition \ref{RWRE_ERW_left}.

\begin{proof}[Proof of Proposition \ref{RWRE_ERW_left}]
Let $\mathbb{E}[(\log M_{-1})_+]<\infty$. Lemma \ref{lemma_powerseries} and $(\ref{eq:2})$ yield
\[\sum_{n\in\N}\frac{M_{-n}}{\prod_{j=1}^{n-1}\rho_{-j}}<\infty \quad \mathbb{P}\text{-a.s.}\]
Lemma \ref{lemma2} $(i)$ implies 
$\sum_{n>(0\vee -z)}P_{0,\omega}[A_n^c(z)]<\infty$
and therefore, by (\ref{eq:limsupA_n}) and (\ref{eq_Borel_Cantelli}),
$P_{0}[S_n<z \text{ eventually}]=1$ for all $z\in\Z$. Hence, $\lim_{n\to\infty}S_n= -\infty$ $P_0$-a.s.

Let $C>\mathbb{E}[\log \rho_0]$.
If $\mathbb{E}[(\log M_{-1})_+]=\infty$, then $\sum_{n\in\N}M_{-n} e^{-Cn}=\infty$ $\mathbb{P}$-a.s.\ according to Lemma \ref{lemma_powerseries}.
Therefore,
$\sum_{n>(0\vee -z)}P_{0,\omega}[A_n^c(z)]=\infty$ is obtained $\mathbb{P}$-a.s.\ for all $z\in \Z$ by Lemma \ref{lemma2} $(ii)$.
Hence, $(S_n)_{n\geq0}$ hits each $z\in\Z$ infinitely many times $P_0$-a.s., see (\ref{eq_Borel_Cantelli}), and the proposition follows.
\end{proof}

\section{Excited random walk in random environment}
\label{sec:ERWRE}
Let us now prove our main Theorem \ref{theorem1} about recurrence and transience of a random walk 
in a left-transient random environment with cookies of strength 1.

\begin{proof}
The results from Sections \ref{sec:pos_halfline} and \ref{sec:neg_halfline} are employed even though the environments $\omega$ are not exactly the same. If we consider excursions to the right of 0, we apply Section \ref{sec:pos_halfline}, as long as cookies on negative integers have no influence on the behaviour of the random walker. In the case where the behaviour of the walker on the negative integers is studied we use Section \ref{sec:neg_halfline} since the cookies on the right side of 0 can be neglected in that case.

Let $\mathbb{E}[(\log M_0)_+]<\infty$.
Proposition \ref{RWRE_ERW_left} $(i)$ gives us the following for $\mathbb{P}$-a.e.\ $\omega$. If we remove all cookies on the positive integers then $\lim_{n\to\infty}S_n=-\infty$ $P_{0,\omega}$-a.s.
On the other hand, $\mathbb{E}[(\log M_0)_+]<\infty$ implies $\limsup_{t\to\infty}(t\cdot\mathbb{P}[\log{M_0}>t])=0$ since 
\begin{align*}
 \mathbb{E}[(\log M_0)_+]&=\int_0^{\infty}\mathbb{P}[\log{M_0}>t]dt\\
&\geq \sum_{n\in\N}(t_n-t_{n-1})\cdot \mathbb{P}[\log{M_0}>t_n]\\
&\geq \frac{1}{2}\sum_{n\in\N}t_n\cdot \mathbb{P}[\log{M_0}>t_n]
\end{align*}
for every sequence $(t_n)_{n\geq1}$ with $t_n-t_{n-1}\geq \frac{1}{2} t_n >0$. Hence, every excursion to the right is $P_{0,\omega}$-a.s.\ finite for $\mathbb{P}$-a.e.\ $\omega$ by Proposition \ref{theorem:ERWRE_right} $(ii)$, Lemma \ref{lem:Divergenz} and Lemma \ref{lem:finite_excursion}. Therefore, all cookies on positive sites that are visited by the random walker are removed. Consequently, the ERWRE tends to $-\infty$ $P_{0,\omega}$-a.s.\ for $\mathbb{P}$-a.e.\ $\omega$.

We turn to the second case. Excursions to the left of 0 depend only on the environment left of 0, i.e.\ on $((\omega(x,i))_{i\in\N})_{x\in -\N}$. Since $\mathbb{E}[(\log M_0)_+]=\infty$, every excursion to the left is $P_0$-a.s.\ finite and the ERWRE returns to 0 by Proposition \ref{RWRE_ERW_left} $(ii)$. On the other hand, excursions to the right only depend on the environment right of 0. Since $\limsup_{t\to\infty}(t\cdot\mathbb{P}[\log{M_0}>t])<\mathbb{E}[\log(\rho_0)]$, every excursion to the right is $P_0$-a.s.\ finite by Proposition \ref{theorem:ERWRE_right} $(ii)$, Lemma \ref{lem:Divergenz} and Lemma \ref{lem:finite_excursion}. Hence, the process returns to 0 infinitely many times $P_0$-a.s.

In the last case, we have $\mathbb{E}[(\log M_0)_+]=\infty$. Since cookies only enforce the drift to the right $\mathbb{E}[(\log M_0)_+]=\infty$ implies that $P_0[\limsup_{n\to\infty} S_n=+\infty]=1$ by Proposition \ref{RWRE_ERW_left} $(ii)$. By the assumption $\liminf_{t\to\infty}(t\cdot\mathbb{P}[\log{M_0}>t])>\mathbb{E}[\log(\rho_0)]$ and by Proposition \ref{theorem:ERWRE_right} $(i)$, we obtain $P_0[S_n\to +\infty]>0$. Furthermore, we also have $P_0[S_n\to +\infty, B]>0$, where $B:=\{S_n >S_0 \;\forall n\geq 1\}$.\\
Applying L\'evy's 0-1 law, we get with the canonical filtration $(\mathcal{F}_n)_{n\geq 0}$ of $(S_n)_{n\geq 0}$
for $\mathbb{P}$-a.e.\ $\omega$
\begin{equation}
\label{eq:Levys}
 P_{0,\omega}[S_n\to\infty|\mathcal{F}_{T_n}]\overset{n\to\infty}{\longrightarrow}\mathbbmss{1}_{\{S_n\to\infty\}}\quad P_{0,\omega}\text{-a.s.}
\end{equation}
Furthermore, we see that for all $n\in\N$
\begin{align}
\label{eq:Levys_calc}
 P_{0,\omega}[S_n\to\infty|\mathcal{F}_{T_n}]\geq P_{0,\omega}[S_n\to\infty, S_{T_n +k}>n \;\forall k\in\N|\mathcal{F}_{T_n}]
=P_{n,\omega}[S_n\to \infty, B].
\end{align}
The environment $(\omega(z, \cdot))_{z\in\Z}$ is i.i.d.\ under $\mathbb{P}$. Thus, the ergodic theorem yields
\begin{equation}
\label{eq:ergodic}
 \frac{1}{m+1} \sum_{n=0}^{m}P_{n,\omega}[S_n\to +\infty, B]\overset{m\to\infty}{\longrightarrow}\mathbb{E}[P_{0,\omega}[S_n\to +\infty, B]]>0 \quad \mathbb{P}\text{-a.s.}
\end{equation}
As a consequence, we get for $\mathbb{P}$-a.e.\ $\omega$ that $P_{n,\omega}[S_n\to +\infty, B]>\epsilon$ for infinitely many $n$ and for some $\epsilon >0$.
Finally, (\ref{eq:Levys}), (\ref{eq:Levys_calc}) and (\ref{eq:ergodic}) yield that $\mathbbmss{1}_{\{S_n\to\infty\}}=1$ $P_{0,\omega}$-a.s.\ for $\mathbb{P}$-a.e.\ $\omega$ and therefore, $P_0[S_n\to +\infty]=1$.
\end{proof}

\begin{remark}
Theorem \ref{theorem1} is still correct if $\mathbb{P}[M_0=\infty]>0$ and $\mathbb{P}[M_0=0]>0$ instead of Assumption \ref{as_A4}.
\end{remark}

In conclusion, we give some examples for the three cases of Theorem \ref{theorem1}. They have been motivated by and adapted from \cite[Theorem 1]{ZeeviGlynn}.

\begin{example}
 Let the assumptions of Theorem \ref{theorem1} be satisfied and require $\lambda, \beta >0$. We set 
\begin{align*}
\mathbb{P}[M_0\geq k]&:=\frac{1}{(1+\beta \log k)^{\lambda}}\quad \text{ for } k\geq2,\; k\in\N, \\
\mathbb{P}[M_0=1]&:=0,\\
\mathbb{P}[M_0=0]&:=1-\frac{1}{(1+\beta \log 2)^{\lambda}}.
\end{align*}

Now the following cases can be derived.
\begin{itemize}
 \item[(i)] If $\lambda >1$, then $P_0[\lim_{n\to\infty}S_n= -\infty]=1$.
 \item[(ii)] If $\lambda <1$, then $P_0[\lim_{n\to\infty}S_n= +\infty]=1$.
 \item[(iii)] If $\lambda =1$ and $\beta\cdot \mathbb{E}[\log \rho_0]<1$, then $P_0[\lim_{n\to\infty}S_n= +\infty]=1$. If $\lambda =1$ and $\beta\cdot \mathbb{E}[\log \rho_0]>1$, then $P_0[S_n= 0 \text{ infinitely often}]=1$.
\end{itemize}
\end{example}

\begin{proof}
These results are due to Theorem \ref{theorem1} and due to the following calculations.\\
In the first case we have $\mathbb{E}[(\log M_0)_+]\leq \sum_{k\in\N_0}\mathbb{P}[\log M_0\geq k]<\infty$ since $\lambda>1$.\\
To prove the second case, note that $\mathbb{P}[\log M_0\geq t]\geq (1+\beta \log (e^t+1))^{-\lambda}$ for $t\geq 0$ and therefore $\liminf_{t\to\infty}t\: \mathbb{P}[\log M_0\geq t]=\infty$ if $\lambda<1$.\\
Finally, if $\lambda =1$ we obtain that $\mathbb{E}[(\log M_0)_+]\geq \sum_{k\in\N}\mathbb{P}[\log M_0\geq k]=\infty$ and $\lim_{t\to\infty}t\: \mathbb{P}[\log M_0\geq t]=\frac{1}{\beta}$. According to the value of $\beta$ the ERWRE tends to $+\infty$ or is recurrent $P_0$-a.s.
\end{proof}

\section*{Acknowledgements}

This work was funded by ERC Starting Grant 208417-NCIRW.\\
Many thanks go to my advisor Prof.\ Martin P.\ W.\ Zerner for his helpful suggestions and his support. I thank Elmar Teufl and Johannes Rue{\ss}
for inspiring discussions.
Furthermore, let me mention here that this paper is an extension of my diploma thesis, written at Eberhard Karls Universit\"at T\"ubingen.

\bibliographystyle{amsplain} 
\bibliography{bibfile}

\parindent=0pt 

\end{document}